\documentclass[11pt]{article}

\usepackage[english]{babel}
\usepackage[utf8x]{inputenc}
\usepackage{amsthm}
\usepackage{amsmath}
\usepackage{amssymb}
\usepackage{graphicx}
\usepackage{epstopdf}
\usepackage[margin=1.0in, marginparwidth=1.2cm, marginparsep=0.3cm]{geometry}
\usepackage{hyperref}

\newtheorem{theorem}{Theorem}

\newtheorem{lemma}{Lemma}
\newtheorem{prop}{Proposition}
\newtheorem*{theorem*}{Theorem}

\theoremstyle{definition}

\newtheorem*{nota}{Notation}

\theoremstyle{remark}
\newtheorem{rem}{Remark}

\newtheorem{conj}{Conjecture}

\def\pr{{\partial}}

\def\V{\Vert}

\begin{document}

\title{Ill-posedness for the incompressible Euler equations in critical Sobolev spaces}

\author{Tarek Mohamed Elgindi and In-Jee Jeong}

\maketitle

\begin{abstract}
	For the $2D$ Euler equation in vorticity formulation, we construct localized smooth solutions whose critical Sobolev norms become large in a short period of time, and solutions which initially belong to $L^\infty \cap H^1$ but escapes $H^1$ immediately for $t>0$. Our main observation is that a localized chunk of vorticity bounded in $L^\infty \cap H^1$ with odd-odd symmetry is able to generate a hyperbolic flow with large velocity gradient at least for a short period of time, which stretches the vorticity gradient.
\end{abstract}

\section{Introduction and the Main Statement}

We consider the vorticity formulation of $2D$ incompressible Euler equation on the torus $\mathbb{T}^2 = [-1,1)^2$: \begin{equation}\label{eq:Euler}
\pr_t \omega + u \cdot \nabla \omega = 0,
\end{equation}
where the velocity $u$ is determined from $\omega$ by the Biot-Savart law \begin{equation}\label{eq:BS}
u(t,x) = \frac{1}{2\pi} \sum_{n \in \mathbb{Z}^2} \int_{[-1,1)^2} \frac{(x-y-2n)^\perp}{|x-y-2n|^2} \omega(t,y)dy,
\end{equation}
with the convention $(v^1,v^2)^\perp = (-v^2,v^1)$. It is well known that \eqref{eq:Euler} is well-posed in $C([0,\infty);H^s(\mathbb{T}^2))$ for $s > 1$ and in $C^\infty([0,\infty)\times \mathbb{T}^2)$ when the initial data is smooth. If $\omega(t,\cdot) \in H^s(\mathbb{T}^2)$ for $s>1$ then $u(t,\cdot)$ is Lipschitz in $x$ and the following ordinary differential equation
 \begin{equation}\label{eq:flowdef}
\begin{split}
\frac{d}{dt} \Phi(t,x) &= u(t,\Phi(t,x)) \quad \mbox{and}\quad \Phi(0,x) = x
\end{split}
\end{equation}
defines the area-preserving flow maps $\Phi(t,\cdot):\mathbb{T}^2 \rightarrow \mathbb{T}^2$ along which the vorticity is transported; \begin{equation}\label{eq:vt}
\omega(t,\Phi(t,x)) = \omega_0(x).
\end{equation} 

In this note, we will provide a simple proof of the \textit{strong illposedness} of the Euler equation \eqref{eq:Euler} in the critical Sobolev space $H^1(\mathbb{T}^2)$, which was obtained recently by Bourgain and Li \cite{MR3359050}. The following result shows that so-called \textit{norm inflation} occurs in $H^1$ for smooth initial data.

\begin{theorem}\label{thm:growth}
	For any $\epsilon > 0$, there exists an initial data $\omega_0 \in C^\infty(\mathbb{T}^2)$ and a time moment $0 < t < \epsilon$ such that \begin{equation*}
		\Vert \omega_0 \Vert_{H^1 \cap L^\infty} < \epsilon, \quad \mathrm{supp}(\omega_0) \subset B_0(\epsilon) \qquad\mbox{and}\qquad \Vert\omega(t,\cdot)\Vert_{H^1} > \epsilon^{-1},
		\end{equation*}
		where $B_0(\epsilon)$ is the ball of radius $\epsilon$ around the origin.
\end{theorem}

Next, we show that a localized solution which is initially small in $L^\infty \cap H^1$ can  immediately escape $H^1$ for $t > 0$, which have also appeared in \cite{MR3359050}:

\begin{theorem}\label{thm:pat}
	For any $s,p$ such that $sp = 2$ and $1 \le s < 6/5$, there is $\omega_0 \in L^\infty(\mathbb{T}^2) \cap W^{s,p}(\mathbb{T}^2)$ which is $C^\infty$ away from the origin that for any $0 < t_0 \le 1$, 
	\begin{equation*} \mathrm{ess}\mbox{-}\mathrm{sup}_{0<t\le t_0} \Vert \omega(t,\cdot)\Vert_{W^{s,p}} \ge \mathrm{ess}\mbox{-}\mathrm{sup}_{0<t\le t_0} \Vert \omega(t,\cdot)\Vert_{H^1} = +\infty.
	\end{equation*}
\end{theorem}

Therefore, for $1 \le s < 6/5$ and $sp = 2$, there are initial data in $W^{s,p}$ which escapes not only $W^{s,p}$ but also $H^1$ for $t > 0$. Actually our proof of Theorem \ref{thm:pat} can be modified a little bit to prove  ill-posedness statements in $W^{s,p}$ for all values of $1 \le s <2$: there exists $\omega_0 \in L^\infty(\mathbb{T}^2) \cap W^{s,p}(\mathbb{T}^2)$ whose $W^{s,p}$-norm becomes instantaneously infinite for $t > 0$.  

In the work of Bourgain-Li  \cite{MR3359050}, the existence of localized initial data which escapes $H^1$ was obtained by carefully ``patching'' together an infinite sequence $\{ \omega_0^{(n)} \}_{n=1}^\infty$ of $C^\infty$ data whose support becomes smaller but grows in $H^1$ with a larger rate in a shorter period of time as $n \rightarrow \infty$. It is possible that the $C^\infty$ solutions that we construct in Theorem \ref{thm:growth} can be patched together to obtain the desired statement as well. However, this seems to require a rather involved analysis, and we have chosen to establish Theorem \ref{thm:pat} via exhibiting a simple explicit initial data in $L^\infty \cap H^1$, see \eqref{ini:continumm}. 

The problem of well-posedness will not be an issue in the above statements as there is a unique, global-in-time solution of the extended system \eqref{eq:BS}--\eqref{eq:vt} in $L^\infty([0,\infty)\times L^\infty(\mathbb{T}^2))$ (so-called Yudovich solutions) for $\omega_0$ just in $L^\infty(\mathbb{T}^2)$, even though in this case, $u(t,\cdot)$ is only log-Lipschitz in general. A simple proof of this fact may be found in \cite{MR1245492}, for instance.

The space $H^1$ is called \textit{critical} since we barely cannot close the standard energy estimate \begin{equation*}
\frac{d}{dt} \Vert \nabla \omega \Vert_{L^2(\mathbb{T}^2)} \le \Vert \nabla u \Vert_{L^\infty(\mathbb{T}^2)} \Vert \nabla \omega \Vert_{L^2(\mathbb{T}^2)},
\end{equation*}
as $\omega \in H^1(\mathbb{T}^2)$ or even  $\omega \in H^1(\mathbb{T}^2) \cap L^\infty(\mathbb{T}^2)$ does not guarantee that $u$ is Lipschitz. This failure of Lipschitz regularity is at the heart of the possibility of rapid growth of vorticity gradient. It is explicit in the Bahouri-Chemin example \cite{BC}: Take $\omega(x) = 1$ on $[0,1]^2$ and extend it to $[-1,1]^2$ as an odd function in both variables. This defines a stationary solution of $2D$ Euler in the sense of equations \eqref{eq:BS}--\eqref{eq:vt}, and the flow near the origin is ``hyperbolic'' in the following specific sense: for $0<x_1<x_2$ small, it can be computed that for some absolute constant $c > 0$ (see Denissov \cite{D}) \begin{equation}\label{eq:BC}
\begin{split}
u(x_1,x_2) = c \left( -x_1 \left( \ln \frac{1}{x_2} +r_1(x_1,x_2) \right) , x_2 \left(  \ln \frac{1}{x_2} + r_2(x_1,x_2) \right) \right)
\end{split}
\end{equation} with some smooth functions $r_1, r_2$. 

Certain perturbations of this stationary solution were utilized in the works of Denissov \cite{D}, Kiselev-{\v{S}}ver{\'a}k \cite{MR3245016}, and Zlato{\v{s}} \cite{Z} (in chronological order) to obtain growth of vorticity gradient in the maximum norm $L^\infty$. The growth rates of $\Vert \nabla\omega\Vert_{L^\infty(\mathbb{T}^2)}$  obtained in \cite{D} and \cite{Z} were double exponential  for arbitary long but finite time and exponential for all time, respectively. The groundbreaking work \cite{MR3245016} settled the possibility of double exponential growth of $\Vert \nabla\omega\Vert_{L^\infty(D)}$ for all time, when the domain is a disc. The ``Key Lemma'' of Kiselev and {\v{S}}ver{\'a}k (which was also utilized in \cite{Z}) is an essential tool in our arguments as well (see below Lemma \ref{lem:Zlatos}). 

While our basic strategy to obtain growth of $\omega$ in $H^1(\mathbb{T}^2)$ is similar to that of the aforementioned works, there are a number of notable differences in our setting. 

First, while the idea of ``linearizing'' around the Bahouri-Chemin stationary solution makes sense when considering only bounded vorticities, this solution does not belong to $H^1(\mathbb{T}^2)$. Hence, we needed to consider a different type of ``background'' vorticity, and our choice was to take a suitably localized version of the following function: \begin{equation*}
\begin{split}
\omega_0(x_1,x_2) = \Delta \left( x_1x_2 \left| \ln |x| \right|^\alpha  \right)
\end{split}
\end{equation*} with $0<\alpha<1/2$. The advantage of this initial vorticity is that it belongs to $H^1\cap L^\infty$ and the corresponding velocity $u_0$ satisfies $\nabla u_0 \notin L^\infty$.

Second, since we want localized solutions, it is not clear if the specific hyperbolic picture of the type \eqref{eq:BC} near the origin will be sustained, even for a very short periodic of time. In view of this, our strategy is to take an initial vorticity which extends over two different length scales $N^{-1}$ and $N^{-1/2}$, and to show that vorticity outside the $O(N^{-1})$-region, in the special time scale of $\ln\ln N/\ln N$, is sufficient to generate a hyperbolic flow which stretches the vorticity gradient on the $O(N^{-1})$-region.

Here, a caveat is that we could not exclude the possibility of our initial vorticity chunk getting ``squeezed'' in the angular direction even earlier than the scale $\ln\ln N/\ln N$, in which case we do not have a good lower bound on $|\nabla u|$. Hence our actual proof is based on a contradiction argument. This difficulty vanishes when the domain has a boundary: see Remark \ref{rem:disc}.

Closing the introduction, let us mention that similar ill-posedness statements were recently established for the integer based $C^k$ spaces with $k\ge1$ of the velocity field $u$, independently in the works of Elgindi-Masmoudi \cite{EM} and Bourgain-Li \cite{MR3320889}. We refer the interested readers to the works \cite{MR3359050}, \cite{MR3320889}, and \cite{EM} for an extensive list of references on the problem of well-posedness of the Euler equations.

\begin{nota}
	Let us use the notation $|f|_p = \Vert f \Vert_{L^p(\mathbb{T}^2)}$ for $p \in [1,\infty]$ for simplicity. We use letters $C, C_1, c,\cdots$ to denote various absolute positive constants, and their values may change from line to line. When a constant depends on some parameters, we explicitly indicate dependence as subscripts. We use superscripts to refer to components of a vector: for example, $u = (u^1, u^2)$ and $\Phi = (\Phi^1, \Phi^2)$.
\end{nota}

\section{Proof of Theorem \ref{thm:growth}}

Our initial vorticity $\omega_0$ will be odd both in the variables $x_1$ and $x_2$. Since this symmetry persists for all time, we may view $\omega(t,\cdot)$ as defined just on $[0,1]^2$.  Pick a large integer $N$, and let us define the initial vorticity on $[0,1]^2$ as follows: \begin{equation}\label{Ini_vorticity}
	\omega_0(r,\theta) := \chi(r) \psi(\theta),
\end{equation}
where $\chi$ and $\psi$ are smooth bump functions. More specifically, they satisfy \begin{equation*}
\chi(r) := \begin{cases}
1 &\mbox{for }  r \in \left[N^{-1}, N^{-1/2} \right] \\
0 &\mbox{for }  r \notin  \left[N^{-1}/2, 2N^{-1/2}\right],
\end{cases} \qquad \mbox{and} \qquad \psi(\theta) := \begin{cases}
1 &\mbox{for }  \theta \in \left[\pi/4, \pi/3\right] \\
0 &\mbox{for }  \theta \notin  \left[\pi/6, 5\pi/12\right].\end{cases}
\end{equation*}
Since \begin{equation*}
| \nabla \omega_0|_2^2 = \int\int r |\pr_r \omega_0|^2 drd\theta  + \int \int \frac{1}{r} |\pr_\theta \omega_0|^2 drd\theta,
\end{equation*}
the main contribution of $| \nabla \omega_0|_2$ comes from the angular variation: $|\nabla \omega_0|_2 \approx c (\ln N)^{1/2}$ as $N \rightarrow \infty$.

As mentioned in the introduction, we need to work with a special time scale. Given $\tau^* > 0$ and $N$, we set $t^*(\tau,N) = \tau^* \ln\ln N/\ln N$ and we shall track the evolution of initial data \eqref{Ini_vorticity} on the time interval $[0,t^*]$. To get an idea of how this scale appears, recall that the main idea is to stretch vorticity in the $O(N^{-1})$ region using the chunk of vorticity ``behind''. Since initially $|\nabla\omega_0|_{L^2(O(N^{-1}))} =O(1)$ while $|\nabla \omega_0|_2 \approx c (\ln N)^{1/2}$, we need to stretch the $H^1$-norm in the local region by a factor of $(\ln N)^{1/2+\epsilon}$ to obtain norm inflation. In view of $|\nabla u_0|_{\infty} \approx c\ln N$, we achieve this goal once we sustain this lower bound on the velocity gradient during an interval of time $[0,t^*]$. It is important that in this time scale, fluid particles can move only up to a factor of $\ln N$, see \eqref{eq:flow} and \eqref{eq:radial_bound} below. 

Our main technical tool is the following expression for the velocity due to Kiselev and {\v{S}}ver{\'a}k \cite{MR3245016}; we use a version by Zlato{\v{s}} \cite[Lemma 2.1]{Z} which works in the case of the torus $\mathbb{T}^2=[-1,1)^2$. 

\begin{lemma}[Key Lemma]\label{lem:Zlatos}
	Let $\omega(t,\cdot)$ be odd in $x_1$ and $x_2$. Then for $x \in [0,1/2)^2$, we have \begin{equation}\label{eq:Zlatos}
	\frac{u^i(t,x)}{x_i} = (-1)^i \frac{4}{\pi} \int_{[2x_1,1) \times [2x_2,1)} \frac{y_1 y_2}{|y|^4} \omega(t,y) dy + B_i(t,x)
	\end{equation}
	with $|B_i| \le C|\omega|_{\infty}\left( 1+ \ln (1+ x_{3-i}/x_i) \right)$ for $i \in \{ 1,2\}$.
\end{lemma}

There are several striking features of this lemma, which we would like to emphasize. First, the expression \eqref{eq:Zlatos} essentially gives a \textit{pointwise} control over the velocity gradient, just under the assumption that $\omega(t,\cdot) \in L^\infty$. It is surprising that such a control is available, especially because the formula is applicable even in situations where $\nabla u$ is unbounded. Next, the integral in \eqref{eq:Zlatos} is \textit{monotone} in $\omega(t,\cdot)$, so that for the purpose of obtaining a lower bound on the velocity gradient, it suffices to find a region in space where vorticity is uniformly bounded from below. On the other hand, one should note that Lemma \ref{lem:Zlatos} is applicable only when the integral term in \eqref{eq:Zlatos} dominates the remainder term $B_i$.

The following estimates are standard (cf. \cite{MR2768550,MR1245492}) and will play a complementary r\^{o}le of the previous lemma.

\begin{lemma}\label{lem:standard}
	Let $(\omega,u,\Phi)$ to be the solution triple for the $2D$ Euler equations in $\mathbb{T}^2$ with initial data 
	$\omega_0$. The velocity is log-Lipschitz \begin{equation}\label{eq:logLip}
	|u(t,x) - u(t,y)| \le C|\omega_0|_\infty |x-y| \left( 1 +  \ln(4/|x-y|) \right),
	\end{equation}
	and the flow maps $\Phi(t,\cdot):\mathbb{T}^2 \rightarrow \mathbb{T}^2$ for $ 0 \le t \le (C|\omega_0|_{\infty})^{-1}$ satisfy quasi-Lipschitz estimates of the form \begin{equation}\label{eq:flow}
	c |x-y|^{\exp(ct|\omega|_\infty)} \le |\Phi(t,x)- \Phi(t,y)| \le C |x-y|^{\exp(-Ct|\omega|_\infty)}.
	\end{equation}
\end{lemma}

Note that the argument of the logarithm in \eqref{eq:logLip} is always greater than 1 since $|x-y| \le \sqrt{2}$ in our torus $[-1,1)^2$. 

\begin{proof}
	Although these estimates are well-known, we provide a proof of \eqref{eq:flow} (assuming the bound in \eqref{eq:logLip}), as it appears throughout the arguments given below. 
	
	For simplicity, we set $d(t) := |\Phi(t,x) - \Phi(t,y)|$, and from the definition of flow we have \begin{equation*}
	\begin{split}
	\frac{d}{dt} \left( \Phi(t,x) - \Phi(t,y) \right) = u(t,\Phi(t,x)) - u(t,\Phi(t,y)),
	\end{split}
	\end{equation*} and applying the estimate \eqref{eq:logLip} gives a bound \begin{equation*}
	\begin{split}
	\left|\frac{d}{dt} d(t) \right| \le C|\omega_0|_{\infty} d(t) \left(1+ \ln \frac{4}{d(t)} \right),
	\end{split}
	\end{equation*} which implies \begin{equation*}
	\begin{split}
	\left|  \frac{d}{dt} \ln \frac{4}{d(t)}   \right| \le C|\omega_0|_{\infty} \left(1+  \ln \frac{4}{d(t)}   \right).
	\end{split}
	\end{equation*} Denoting $f(t)$ and $g(t)$ as the unique solution of the respective ODE system \begin{equation*}
	\begin{split}
	  \frac{d}{dt} \ln \frac{4}{f(t)}  = C|\omega_0|_{\infty} \left(1+  \ln \frac{4}{f(t)}   \right),\qquad \frac{d}{dt} \ln \frac{4}{g(t)}  =- C|\omega_0|_{\infty} \left(1+  \ln \frac{4}{g(t)}   \right)
	\end{split}
	\end{equation*} on the time interval $[0,(C|\omega_0|_{\infty})^{-1}]$ with initial data $f(0) = g(0) = d(0) = |x-y|$, we obtain the desired estimates as $g(t) \le d(t) \le f(t)$. 
\end{proof}

\begin{figure}
	\includegraphics[height=60mm]{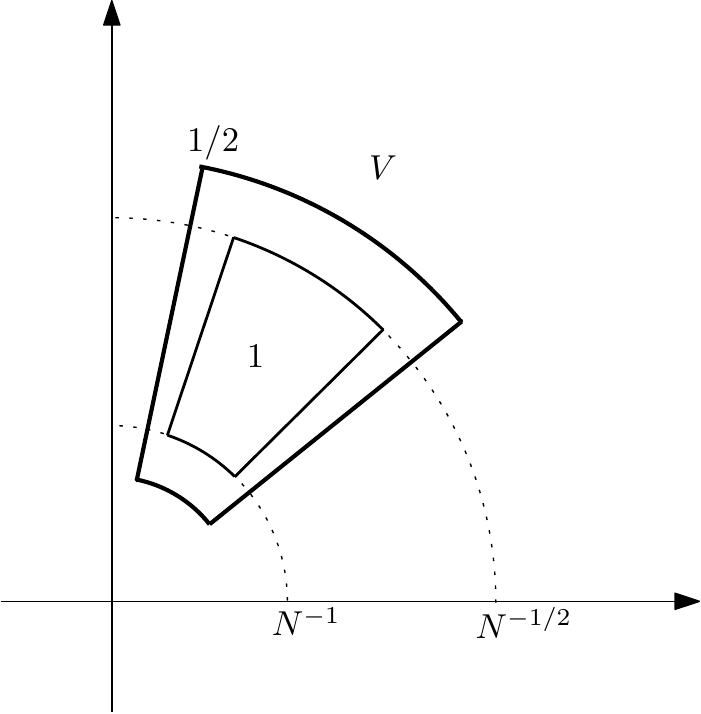} \qquad
	\includegraphics[height=60mm]{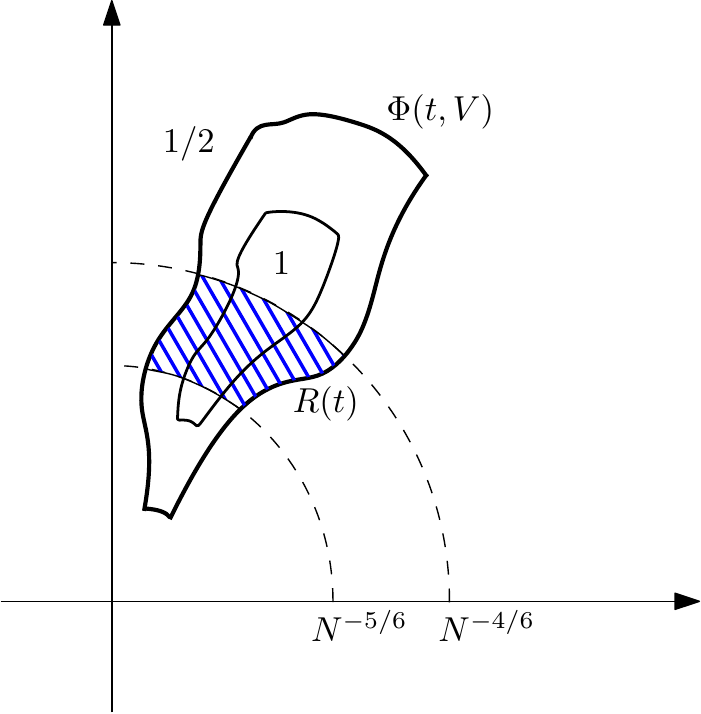}
	\centering
	\caption{The figure on the left describes the initial data $\omega_0^{(N)}$, where the two rectangles represent the region where $\omega_0^{(N)} \equiv 1$ (inner) and $V$ (outer). The right figure shows possible evolution of the set $V$ under the flow. The shaded region represents $R(t)$.}
	\label{fig:rectangles}
\end{figure}

Given lemmas above, we present the proof of Theorem \ref{thm:growth}.

\begin{proof}[Proof of Theorem \ref{thm:growth}]
	We will instead show the following statement:
	
	\medskip
	
	\textbf{Claim.} For any $M > 0$, there exists some $N_0, \tau^* > 0$ depending only on $M$ such that for all $N \ge N_0$, the solution associated with the initial data as in \eqref{Ini_vorticity} satisfies, with an absolute constant $C$, \begin{equation}\label{eq:goal}
	\frac{|\nabla \omega(t_N,\cdot)|_2}{|\nabla\omega_0|_2} \ge CM^{1/2} \quad\mbox{for some}\quad  0 < t_N \le \tau^* \frac{\ln\ln N}{\ln N}.
	\end{equation}
	
	Once it is established, we simply use the scaling symmetry of the Euler equation: given a solution $\omega(t,x)$ and $\lambda > 0$, $\omega^\lambda(t,x):=\lambda\omega(\lambda t,x)$ is another solution with initial data $\lambda \omega_0$, and we can pick $\lambda = (\ln N)^{-1/2} M^{-1/4}$ to achieve the statements of the theorem.
	
	\medskip
	
	Given $M > 0$, we fix $\tau^* = \alpha_M$, where $\alpha_M>0$ is a constant depending only on $M$ to be defined below. In several places of the following argument, it is implicitly assumed that $N$ is sufficiently large with respect to $M$ and some absolute constants appearing in the proof.
	
	Consider the annulus $A = \{ r : N^{-5/6} \le r \le N^{-4/6} \}$. During the time interval $[0,t^*]$, particles starting from the arc $\{ (r,\theta): r  = N^{-1}, \pi/4 \le \theta \le \pi/3 \}$ remain in the region $\{ r < N^{-5/6} \}$ under the flow $\Phi(t,\cdot)$. Similarly, particles from $\{ (r,\theta): r  = N^{-1/2}, \pi/4 \le \theta \le \pi/3 \}$ cannot escape $\{ r > N^{-4/6} \}$. Both statements follow from \eqref{eq:flow} applied with $y = 0$ and $|x| = N^{-m}$ (where $1/2 \le m \le 1$): we have \begin{equation}\label{eq:centerflow}
	c|x|^{\exp(ct|\omega|_\infty)} \le |\Phi(t,x)| \le C|x|^{\exp(-Ct|\omega|_\infty)},
	\end{equation}
	and since $t = \tau \ln \ln N / \ln N$ for some $ 0 \le \tau \le \tau^*$, we obtain \begin{equation}\label{eq:radial_bound}
	c(\ln N)^{c\tau} \le \frac{|\Phi(t,x)|}{|x|} \le C(\ln N)^{C\tau}
	\end{equation}
	with constants $c,C > 0$ uniform over $1/2\le m\le 1$. In particular, it implies that any line segment $\{ (r,\theta_0): N^{-1} \le r \le N^{-1/2} \}$ should evolve in a way that it intersects each circle $\{ r = r_0 \}$ for $N^{-5/6} \le r_0 \le N^{-4/6}$. 
	
	Take the domain $V := \{ (r,\theta): \omega_0(r,\theta) \ge 1/2 \}$ and consider the region \begin{equation*}
	R(t) := \Phi(t,V) \cap A
	\end{equation*}
	which is a curvilinear rectangle whose two opposite edges are bounded by $A$ (see Figure \ref{fig:rectangles}). Note that $\omega(t,\cdot) \ge 1/2$ on $R(t)$. For each $r_0 \in [N^{-5/6},N^{-4/6}]$, consider the closed set \begin{equation*}
	I(t,r_0):=\{ 0 \le \theta \le \pi/2 : (r_0,\theta) \in R(t)\},
	\end{equation*} and let us denote its Lebesgue measure by $|I(t,r_0)|$. To show that the $H^1$-norm grows, we are led to consider two different scenarios.
	
	\medskip
	
	\textbf{Case I}. Assume that there exists a time moment $0< t_{cr} \le t^*$ such that for more than half (with respect to the Haar measure $r^{-1}dr$) of $r_0 \in [N^{-5/6},N^{-4/6}]$, we have $|I(t_{cr},r_0)| \le M^{-1}$.
	
	\medskip
	
	If $r_0 \in [N^{-5/6},N^{-4/6}]$ is such that $|I(t_{cr},r_0)| \le M^{-1}$, then we can definitely pick some $\theta_0 = \theta_0(r_0)$ such that the points $(r_0,\theta_0)$ and $(r_0,\theta_0 + \delta)$ with $0 < \delta \le M^{-1}$ satisfy $\omega(t_{cr},r_0,\theta_0) = 1$ and $\omega(t_{cr},r_0,\theta_0+\delta) = 1/2$. This implies a lower bound \begin{equation*}
	M^{-1/2} \left( \int_{I(t_{cr},r_0)}|\pr_\theta\omega(t_{cr},r_0,\theta)|^2 d\theta \right)^{1/2} \ge  \int_{\theta_0}^{\theta_0+\delta} \pr_\theta\omega(t_{cr},r_0,\theta) d\theta = 1/2
	\end{equation*}
	which in turn gives that \begin{equation*}
	|\nabla \omega(t_{cr},\cdot)|_2^2 \ge \int_{N^{-5/6}}^{N^{-4/6}} \int_{I(t_{cr},r)} |\pr_\theta\omega(t_{cr},r,\theta)|^2 d\theta \frac{dr}{r} \ge CM \ln N,
	\end{equation*}
	with some absolute constant $C > 0$. We have established the \textbf{Claim} in this case, recalling that $|\nabla\omega_0|_2^2\le C(\ln N)$.
	
	\medskip
	
	\textbf{Case II}. For \textit{all} $t\in [0,t^*]$, for at least half (again with respect to the  measure $r^{-1}dr$) of $r_0 \in [N^{-5/6},N^{-4/6}]$, we have $|I(t,r_0)| \ge M^{-1}$.
	
	\medskip
	
	In this scenario, we will track the evolution of the following segment \begin{equation*}
	S  = \{ (h,h): N^{-1} \le h \le (\ln N)^{K\tau^*} N^{-1}  \}
	\end{equation*}
	for the time interval $[0,t^*]$, where $K>0$ is an absolute constant to be determined below. Since $\omega(t,\Phi(t,S)) \equiv 1$ for all $t \ge 0$, to show growth of the $H^1$-norm of $\omega$, it is enough to demonstrate that $\Phi(t^*,S)$ is close enough to the vertical segment (where $\omega$ vanishes). In the remaining part of the proof, we will always assume that $x \in S$ and $t \in [0,t^*]$. 
	
	As a first step, we collect simple bounds on the trajectory of $x = (h,h) \in S$, which will in particular guarantee the applicability of Lemma \ref{lem:Zlatos}. To begin with, applying \eqref{eq:flow} with $y = (0,0)$ gives \begin{equation*}
	\Phi^2(t,x) \le |\Phi(t,x)| \le h (\ln N)^{C\tau^*}
	\end{equation*} (recall that $t^*$ and $\tau^*$ are related by $t^* = \tau^* \ln \ln N/\ln N$). Next, to obtain a lower bound on $\Phi^1(t,x)$, we use the log-Lipschitz estimate: \begin{equation*}
	|u^1(t,\Phi(t,x))| = |u^1(t,\Phi(t,x))-u^1(t,0,\Phi^2(t,x))| \le C\Phi^1(t,x) \left( 1+ \ln \frac{4}{\Phi^1(t,x)} \right)
	\end{equation*} and since \begin{equation*}
	\begin{split}
	\frac{d}{dt} \Phi^1(t,x) = u^1(t,\Phi(t,x)),
	\end{split}
	\end{equation*} proceeding exactly as in the proof of the estimate \eqref{eq:flow} of Lemma \ref{lem:standard} gives that \begin{equation*}
	\Phi^1(t,x) \ge \frac{h}{(\ln N)^{C\tau^*}}.
	\end{equation*} Hence, for $x \in S$, we have \begin{equation*}
	\frac{\Phi^2(t,x)}{\Phi^1(t,x)} \le C(\ln N)^{2C\tau^*}
	\end{equation*}
	for $0<t<t^*$.
	
	On the other hand, with our assumption on $|I(t,r)|$, we estimate the integral appearing in Lemma \ref{lem:Zlatos} at the point $\hat{x} = (N^{-7/8},N^{-7/8})$:
	\begin{equation*}
	\begin{split}
	Q(t,\hat{x}) :=  \int_{ [2N^{-7/8},1)^2 } \frac{y_1 y_2}{|y|^4} \omega(t,y)dy \ge \frac{1}{2}\int_{N^{-5/6}}^{N^{-4/6}}\int_{I(t,r)} \frac{\sin\theta \cos\theta}{r} d\theta
	dr\end{split}
	\end{equation*}
	and upon setting \begin{equation*}
	c_M:= \min_{\substack{|I| = 1/2M \\ I \subset [0,\pi/2]}} \int_I \sin\theta\cos\theta d\theta,
	\end{equation*}
	we obtain \begin{equation}\label{eq:integral_lb}
	Q(t,\hat{x}) \ge c_1 c_M \ln N
	\end{equation}
	for some $c_1>0$, whenever $0 \le t \le t^*$. Therefore, we conclude that the $B_i(t,x)$-term can be neglected in Lemma \ref{lem:Zlatos} (by possibly adjusting the value of $c_1$ in \eqref{eq:integral_lb}) as long as we apply it to the trajectory of $S$. That is, 
	\begin{equation*}
	\frac{-u^1(t,\Phi(t,x))}{\Phi^1(t,x)} \ge C Q(t,\hat{x}) \ge C_M \ln N,
	\end{equation*}
	for $x \in S$ and $ 0 \le t \le t^*$ for some constant $C_M > 0$ depending only on $M$. Similarly, we deduce that $u^2(t,\Phi(t,x)) \ge 0$ on the same time interval for $x \in S$. From these bounds, it follows that the curve $\hat{S} := \Phi(t^*,S)$ is contained in the region $$\{ (y_1,y_2): y_2/y_1 \ge (\ln N)^{C_M \tau^*} \}.$$ The flow estimate \eqref{eq:radial_bound} further gives that $\hat{S}$ intersects the circles $\{ r = (\ln N)^{C\tau^*}/N \}$ and $\{ r = (\ln N)^{(K-c)\tau^*}/N \}$. We could have taken $K$ so that $K -c > C$ (where $c,C$ are constants from the estimate \eqref{eq:radial_bound}). Then, for each $ (\ln N)^{C\tau^*}/N \le r \le (\ln N)^{(K-c)\tau^*}/N$, we may find a point (in polar coordinates) of the form $(r,\theta^*(r))$ on $\hat{S}$ such that $\pi/2 - \theta^*(r) \le C(\ln N)^{-C_M \tau^*}$. Therefore, we deduce that \begin{equation*}
	\frac{1}{C}(\ln N)^{C_M \tau^*} \le \int_{0 \le \theta \le \pi/2} |\partial_\theta \omega(t^*,r,\theta)|^2 d\theta 
	\end{equation*}
	and integrating over $(\ln N)^{C\tau^*}/N \le r \le (\ln N)^{(K-c)\tau^*}/N$ against $r^{-1}dr$ with the choice  $\tau^* := 1/C_M$ gives \begin{equation*}
	\begin{split}
	\frac{1}{C} (\ln N)  \ln \left( \frac{K-c-C}{C_M} \ln N    \right) \le \int \int \frac{1}{r}  |\partial_\theta \omega(t^*,r,\theta)|^2 d\theta dr \le |\nabla\omega(t^*)|_2^2
	\end{split}
	\end{equation*} which gives the desired lower bound in \eqref{eq:goal}.
\end{proof}

\begin{rem}\label{rem:disc}
	This construction carries over to the setting of the whole domain and a bounded open set, with minor modifications. 
	
	In the case when the fluid domain is a disc (or more generally, a bounded open set with an axis of symmetry), we can utilize the boundary to achieve Theorem \ref{thm:growth} without relying on a contradiction argument. To be more specific, assume for simplicity that our domain is the upper half-plane $\{ (r,\theta): 0 \le \theta \le \pi \}$. Take $\omega_0$ which is odd in $x_1$ and equals a smoothed out version of the indicator function on the polar rectangle $[N^{-1},N^{-1/2}] \times [0,\pi/4]$ in the positive quadrant. Then it can be shown that for the time interval that we consider, we do not run out of angles; i.e. \textbf{Case I} does not happen. The same can be said for the proof of Theorem \ref{thm:pat}, and actually one can even show continuous-in-time loss of regularity of the solution. We expand on this point in our forthcoming work \cite{EJ}. 
	\end{rem}

\begin{rem}
	Inspecting the proof, one can check that $C_M  = CM^{-2}$ works, and so that we may choose $N \ge CM \exp(CM^2)$ as $M \rightarrow \infty$. In other words, the initial data in \eqref{Ini_vorticity} grows at least by a multiple of $(\ln N)^{1/2 - \epsilon}$ (in both scenarios). Again, when we have a boundary available, it is not necessary to introduce $M$ and we obtain growth by a factor of $(\ln N)^K$ for any $K>0$ as long as $N$ is sufficiently large.
\end{rem}

\section{Proof of Theorem \ref{thm:pat}}
	This time, we consider an odd initial vorticity defined on $[0,1)^2$ by \begin{equation}\label{ini:continumm}
	\omega_0(r,\theta) = \left(\ln \frac{1}{r}\right)^{-\alpha} \psi(\theta)\xi(r),
	\end{equation}
	where $\psi(\cdot)$ is the same angular bump function as in \eqref{Ini_vorticity} and $\xi(r)$ is a smooth bump function which identically equals 1 for $0 \le r \le \epsilon/2$ and vanishes for $r \ge 2\epsilon/3$. Clearly, $\omega_0$ is a bounded continuous function and by choosing $\epsilon > 0$ small enough, we may assume that $|\omega_0|_\infty, |\nabla\omega_0|_2 \le 1$. Given $s$ and $p$ satisfying $sp = 2$ and $1 \le s < 6/5$, we can find a value of $1/2 < \alpha < 3/5$ so that $\Vert \omega_0 \Vert_{W^{s,p}} < + \infty$: note that \begin{equation*}
	\begin{split}
	|\nabla|^s \omega_0(r,\theta) \approx \frac{1}{r^s}\left(\ln \frac{1}{r} \right)^{-\alpha} \psi'(\theta)\xi(r),\qquad r \ll 1,
	\end{split}
	\end{equation*} so that given $sp = 2$, \begin{equation*}
	\begin{split}
	\omega_0 \in W^{s,p}(\mathbb{T}^2) \qquad \mbox{ if and only if } \qquad \alpha p > 1.
	\end{split}
	\end{equation*}
	
	It can be shown that the solution associated with the initial data \eqref{ini:continumm} remains $C^\infty$-smooth away from the origin for all time (see Proposition \ref{prop:smooth} below). Hence, if we denote the solution by $\omega(t,\cdot)$, its $H^1$-norm can be unambiguously defined by \begin{equation*}
	\lim_{\delta \rightarrow 0^+} \int_{ |y| > \delta} |\nabla \omega(t,y)|^2 dy,
	\end{equation*}
	which can take the value $+\infty$.	We will show that there exists a sequence of positive time moments $\{ t_M \}_{M \ge 1}$ and a sequence of radii $\{ r_M \}_{M \ge 1}$, such that $t_M \rightarrow 0^+$, $r_M \rightarrow 0^+$, and for a fixed absolute constant $c>0$, \begin{equation*}
	\int_{|y| > r_M} |\nabla\omega(t_M,y)|^2 dy > cM^{1/2}.
	\end{equation*}
	For each fixed $r > 0$, the function $\int_{|y| > r} |\nabla\omega(t,y)|^2 dy$ is continuous in time and provides a lower bound for $|\nabla\omega(t,\cdot)|_2^2$. Therefore, the existence of sequences satisfying above gives the statement in Theorem \ref{thm:pat}. 
	
	The proof we present is strictly analogous to that of Theorem \ref{thm:growth}, as $\omega_0$ in \eqref{ini:continumm} can be viewed as a ``continuum'' version of data from our previous proof. To be more specific, pick some large number $N$ and radially truncate the function \eqref{ini:continumm} at length scales $N^{-1}$ and $N^{-1/2}$. Then this is essentially a scalar multiple of the smooth initial data $\omega_0^{(N)}$ from the previous section, and recalling the scaling symmetry of the Euler equation, it follows that this truncated initial data grows in $H^1$ by a factor which diverges with $N$ at some time moment $ 0 < t^{(N)}$ which converges to 0 as $N \rightarrow +\infty$. Therefore it is intuitively clear that the data \eqref{ini:continumm} would escape $H^1$ immediately.

\begin{proof}[Proof of Theorem \ref{thm:pat}]
	Given $M > 0$, we consider the time moment \begin{equation*}
	t^* = \tau^* \frac{\ln \ln N}{(\ln N)^{1-5\alpha/3}}
	\end{equation*}
	where $\tau^* = \tau^*(M)$ is to be determined later. It will be implicitly assumed that $N$ is sufficiently large with respect to $M$ and a few absolute constants. In particular, as $1-5\alpha/3 > 0$, it guarantees that $t^* \ll 1$. Throughout the proof, it will be always assumed that the variable $t$ take values in the interval $[0,t^*]$.

	The outline of the argument is as follows: we identify a ``bulk'' region which initially extends over length scales $N^{-1/2}$ and $O(1)$, and a ``local'' region near $N^{-1}$. In the special time interval that we consider, if there is too much angular squeezing of the bulk, then we are done. Otherwise, the bulk region has enough mass which stretches vorticity in the local region. We note in advance that, compared to the situation of Theorem \ref{thm:growth}, we have less precise information on the local particle trajectories, so we should apply Lemma \ref{lem:Zlatos} in a very careful manner.


	To begin with, using the basic estimate \eqref{eq:centerflow} (recall that $|\omega|_\infty \le 1$), we take $0< a_0 < \epsilon/2$ such that the fluid particles starting from $|x| > a_0$ at $t = 0$ cannot cross the circle $|x| = a_0/2$ within $[0,t^*]$, as $t^*$ can be taken to be much smaller than a few absolute constants. Similarly using the same estimate, we can ensure that the particles starting on the circle $|x| = N^{-5/10}$ cannot escape the annulus $$\{ N^{-6/10}<|x|< N^{-4/10} \}$$ in the same time interval. Indeed, taking the logarithms of \eqref{eq:centerflow} (assuming $|x|, |\Phi(t,x)|$ small enough), \begin{equation*}
	\begin{split}
	e^{ct} \ln \frac{1}{|x|} - c_1 \ge \ln \frac{1}{|\Phi(t,x)|} \ge e^{-Ct} \ln \frac{1}{|x|} - C_1
	\end{split}
	\end{equation*} so that in the time interval that we consider, $\ln(1/|\Phi(t,x)|)$ is equivalent to $\ln(1/|x|)$ up to absolute constants which can be assumed arbitrarily close to 1, uniformly in $t$ and $|x|$. 
	
	Given these bounds, take the polar rectangle \begin{equation*}
	V := \{ (r,\theta): N^{-1/2} \le r \le a_0, \pi/4 \le \theta \le 3\pi/8 \}
	\end{equation*}
	and consider intersections of the form \begin{equation*}
	R(t):= \Phi(t,V) \cap \{ (r,\theta): N^{-4/10} \le r \le a_0/2\}.
	\end{equation*}
	Note that on the ``angular'' sides of $V$, $\omega_0$ takes the values $(\ln r^{-1})^{-\alpha}$ and $\beta (\ln r^{-1})^{-\alpha}$ respectively, for some $0<\beta<1$. For each $t \in [0,t^*]$ and $r \in [N^{-1/4},a_0/2]$, we consider the (non-empty) set of angles \begin{equation*}
	I(t,r):= \{ 0 \le \theta \le \pi/2 : (r,\theta) \in R(t)  \}.
	\end{equation*}
	
	We again consider two cases; introducing the set of radii with ``enough'' angles \begin{equation*}
	\begin{split}
	A(t) := \{ r \in [N^{-1/4},a_0/2] :  |I(t,r)| \ge M^{-1} \left( \ln \frac{1}{r} \right)^{-\alpha/3}  \}
	\end{split}
	\end{equation*} (note the power $-\alpha/3$) and first, assume that there exists some $0<t_{cr}<t^*$ such that
	\begin{equation*}
	\begin{split}
	\int_{r \in A(t_{cr})}  \left( \ln \frac{1}{r} \right)^{-5\alpha/3} \frac{dr}{r} \le \frac{1}{2} \int_{r \in [N^{-1/4},a_0/2] }\left( \ln \frac{1}{r} \right)^{-5\alpha/3} \frac{dr}{r} \le C \left( \ln N \right)^{1-5\alpha/3}. 
	\end{split}
	\end{equation*} In this case, we argue exactly as \textbf{Case I} of the previous proof: whenever $r \notin A(t_{cr})$, we integrate over angle to get \begin{equation*}
	\begin{split}
	\int_0^{\pi/2} |\partial_\theta \omega(t_{cr},r,\theta)|^2 d\theta \ge C M \left(  \ln \frac{1}{r} \right)^{-5\alpha/3},
	\end{split}
	\end{equation*} where we have used that when $r = |\Phi(t,y)|$,  $\omega(t,\Phi(t,y)) = \omega_0(y) = (\ln |y|^{-1})^{-\alpha}$ for $y$ having the form $(r,\pi/4)$ in polar coordinates and similarly $\omega(t,\Phi(t,y))  = \beta (\ln |y|^{-1})^{-\alpha}$ when $y = (r,3\pi/8)$, and that $\ln(1/|y|)$ and $\ln(1/|\Phi(t,y)|)$ are equivalent up to some absolute constants arbitrarily close to 1 (relative to the difference between $1$ and $\beta$). Integrating the above lower bound over $r \notin A(t_{cr})$ gives the desired estimate 
	\begin{equation*}
	\begin{split}
	\int_{|y| \ge N^{-4/10}} |\nabla\omega(t_{cr},y)|^2 dy \ge CM \int_{[N^{-4/10},a_0/2]\backslash A(t_{cr})}  \left(  \ln \frac{1}{r} \right)^{-5\alpha/3} \frac{dr}{r} \ge cM( \ln N )^{1-5\alpha/3} . 
	\end{split}
	\end{equation*}

	\begin{figure}
		\includegraphics[height=60mm]{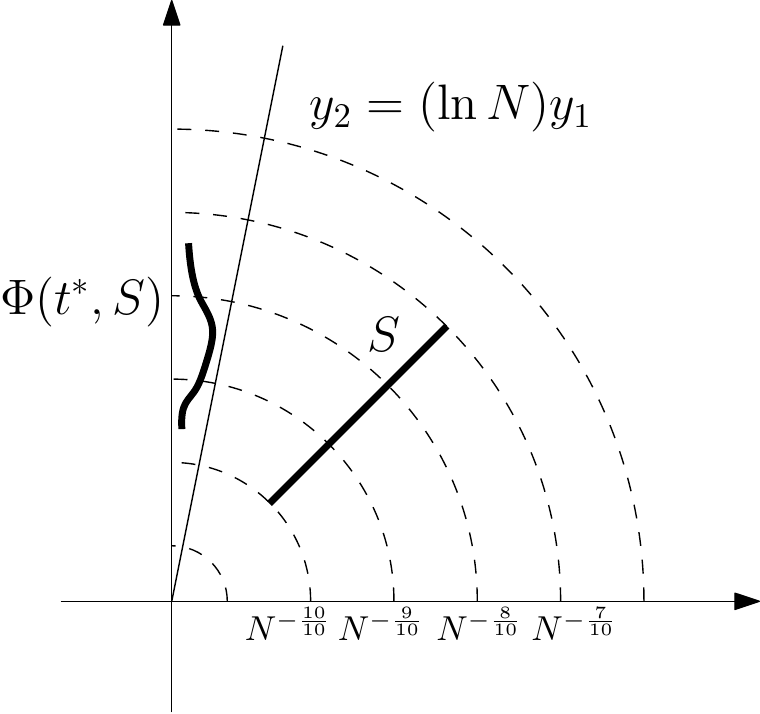}
		\centering
		\caption{Evolution of the local diagonal segment}
		\label{fig:rectangles2}
	\end{figure}


	Therefore, we may assume that for all $t\in [0,t^*]$, we have a lower bound \begin{equation*}
	\begin{split}
	\int_{r \in A(t)} \left( \ln \frac{1}{r}\right)^{-5\alpha/3} \frac{dr}{r} \ge c(\ln N)^{1-5\alpha/3}. 
	\end{split}
	\end{equation*} Under this hypothesis, we shall track the evolution of the diagonal segment in the ``local'' region: \begin{equation*}
	S :=\{ (h,h): N^{-1} \le h \le N^{-7/10} \}.
	\end{equation*} We may assume that the trajectories of the two endpoints of $S$ are trapped in the annuli $\{ N^{-11/10} < r < N^{-9/10} \}$ and $\{ N^{-8/10} < r < N^{-6/10} \}$, respectively. Assume for a moment that we have \begin{equation}\label{eq:containment}
	\begin{split}
	\Phi(t^*,S) \subset \{ (y_1,y_2): y_2 \ge (\ln N) y_1 \},
	\end{split}
	\end{equation} see Figure \ref{fig:rectangles2}. On the set $\Phi(t^*,S)$, $\omega(t^*,\cdot)\ge C(\ln N)^{-\alpha}$, and on the vertical line, $\omega(t^*,\cdot) \equiv 0$. Therefore, for each $N^{-9/10} < r < N^{-8/10}$, we have a lower bound \begin{equation*}
	\begin{split}
	\left| \frac{\pi}{2} - \theta^*(r) \right|^{1/2} \left(  \int |\partial_\theta \omega(t^*,r,\theta)|^2 d\theta  \right)^{1/2} \ge \int_{\theta^*(r)}^{\pi/2} |\partial_\theta \omega(t^*,r,\theta)|d\theta \ge C(\ln N)^{-\alpha},
	\end{split}
	\end{equation*} where $(r,\theta^*(r))$ is a point on $\Phi(t^*,S)$. This gives \begin{equation*}
	\begin{split}
	\int_0^{\pi/2} |\partial_\theta \omega(t^*,r,\theta)|^2 d\theta \ge C(\ln N)^{1-2\alpha}, 
	\end{split}
	\end{equation*} and integrating over $N^{-9/10} < r < N^{-8/10}$ against $r^{-1}dr$, \begin{equation*}
	\begin{split}
	\int_{|y| \ge N^{-1} } |\nabla\omega(t^*,y)|^2 dy \ge \int_{N^{-9/10}}^{N^{-8/10}} \int_0^{\pi/2} |\partial_\theta \omega(t^*,r,\theta)|^2 d\theta \frac{dr}{r} \ge C(\ln N)^{2(1-\alpha)}. 
	\end{split}
	\end{equation*} Since $\alpha <1$, we obtain the desired lower bound. Hence our goal now consists of establishing the containment in \eqref{eq:containment}.

	
	Let us begin by obtaining a lower bound on the integral term appearing in Lemma \ref{lem:Zlatos}. We fix a reference point $\hat{x} = (N^{-6/10},N^{-6/10})$. Then,
	\begin{equation*}
	Q(t,\hat{x}) :=\int_{ [2N^{-6/10},1)^2  } \frac{y_1y_2}{|y|^4} \omega(t,y)dy \ge C \int_{N^{-1/4}}^{a_0/2}   \int_{I(t,r)}  \sin\theta\cos\theta d\theta \left( \ln \frac{1}{r} \right)^{-\alpha} \frac{dr}{r},
	\end{equation*} where we have again used the observation that $\ln|\Phi(t,x)|^{-1}$ and $\ln|x|^{-1}$ are comparable. Now recalling that \begin{equation*}
	\begin{split}
	|I(t,r)| \ge M^{-1} \left( \ln \frac{1}{r} \right)^{-\alpha/3},\qquad r \in A(t),
	\end{split}
	\end{equation*} for $r \in A(t)$ we have an estimate \begin{equation*}
	\begin{split}
	\int_{I(t,r)}  \sin\theta\cos\theta d\theta \ge c M^{-2}  \left(\ln \frac{1}{r} \right)^{-2\alpha/3}
	\end{split}
	\end{equation*} which gives (under our hypothesis on the size of $A(t)$) \begin{equation*}
	\begin{split}
	Q(t,\hat{x}) \ge C_M \int_{A(t)}  \left( \ln \frac{1}{r} \right)^{-5\alpha/3} \frac{dr}{r} \ge C_M (\ln N)^{1-5\alpha/3}. 
	\end{split}
	\end{equation*}
	
	We may now apply Lemma \ref{lem:Zlatos}. From now on, we reserve the letter $x$ for points in the diagonal segment $S$. To begin with, on the diagonal, the error terms $B_1, B_2$ in Lemma \ref{lem:Zlatos} are bounded by an absolute constant, and since $Q(t,\hat{x}) \gg 1$, we can ensure that the trajectory of $x$ stays in the region $\{(y_1,y_2) : y_2 \ge y_1 \}$. Then in turn, this information guarantees that $u^2(t,\Phi(t,x))$ is positive, as $|B_2|$ is bounded by an absolute constant whenever $\{(y_1,y_2) : y_2 \ge y_1 \}$, which gives in particular $\Phi^2(t,x) \ge x_2$. 
	
	Fix some $x \in S$, and assume for the sake of contradiction that $\Phi(t,x)$ is not contained in the region $\{ (y_1,y_2): y_2 \ge (\ln N)y_1 \}$, for \textit{all} $0 \le t \le t^*$. It guarantees that \begin{equation*}
	\begin{split}
	|B_1(t,\Phi(t,x))| \le C\left(1 + \ln\left(1 + \frac{\Phi^2(t,x)}{\Phi^1(t,x)}\right) \right) \le C\ln \ln N
	\end{split}
	\end{equation*} and in particular the $B_1$ term is dominated by $Q(t,\hat{x})$. Hence, \begin{equation*}
	\begin{split}
	\frac{-u^1(t,\Phi(t,x))}{\Phi^1(t,x)} \ge CQ(t,\hat{x}) \ge C_M (\ln N)^{1-5\alpha/3},
	\end{split}
	\end{equation*} and therefore integrating in time over $[0,t^*]$ with the choice $\tau^* = c/C_M$ (for some absolute constant $c>0$), we deduce $\Phi^1(t^*,x) \le x_1 (\ln N)^{-1}$. Combined with $\Phi^2(t^*,x) \ge x_2 = x_1$, \begin{equation*}
	\begin{split}
	\frac{\Phi^2(t^*,x)}{\Phi^1(t^*,x)} \ge \ln N. 
	\end{split}
	\end{equation*} This is a contradiction, so there must exist $0 \le t' \le t^*$ for which \begin{equation*}
	\begin{split}
	\frac{\Phi^2(t',x)}{\Phi^1(t',x)} \ge \ln N. 
	\end{split}
	\end{equation*} However, observe that for any point on the line $y_2 = (\ln N)y_1$, we have $|B_1| \le \ln\ln N$, so that the trajectory of $x$ for $t \ge t'$ cannot escape the region $\{y_2 \ge (\ln N)y_1\}$ unless $|\Phi(t,x)|$ becomes larger than $|\hat{x}|$, which is impossible during the time interval $[0,t^*]$. This finishes the proof. 
\end{proof}

\begin{rem}
	One does not face the restriction $\alpha < 3/5$ in the presence of a boundary.
\end{rem}

For the convenience of the reader, we give a proof that $\omega(t,\cdot)$ in the case of Theorem \ref{thm:pat} actually stays $C^\infty$ away from the origin.

\begin{prop}\label{prop:smooth}
	Consider $\omega_0 \in L^\infty(\mathbb{T}^2)$ which is $C^\infty$ away from a closed set $A \in \mathbb{T}^2$. Then, the unique solution $\omega(t,\cdot)\in L^\infty(\mathbb{T}^2)$ of the 2D Euler equation stays $C^\infty$ away from $\Phi(t,A)$ for all $t > 0$. 
\end{prop}

\begin{proof}
	We may assume that $\V \omega_0 \V_{L^\infty} = 1$. Once we show that $\omega(t,\cdot)$ is smooth away from $\Phi(t,A)$ for $t \in [0,T]$ with some absolute constant $T > 0$ then we may iterate the argument to extend the statement to any finite time moment. 
	
	Take an open set $O$ which is separated from $A$. It suffices to show that for $t \in [0,T]$, there exists some $\alpha > 0$ that $\omega(t,\cdot)$ is uniformly $C^{k,\alpha}$ in $\Phi(t,U_k)$, for any integer $k \ge 0$ and some open set $U_k \supset O$. We deduce this by inducting on $k$. 
	
	For the base case of $k = 0$, take some open set $U_0 \supset \mathrm{cl}(O)$ which is still separated away from $A$. Then we simply write \begin{equation*}
	\begin{split}
	\V\omega(t,\cdot)\V_{C^\alpha(\Phi(t,U_0))} =  \sup_{x, x' \in U_0} \frac{|\omega(t,\Phi(t,x)) - \omega(t,\Phi(t,x'))|}{|\Phi(t,x)-\Phi(t,x')|^\alpha} \le \V \omega_0\V_{C^1(U_0)} \sup_{x,x'\in U_0}  \frac{|x-x'|}{|\Phi(t,x)-\Phi(t,x')|^\alpha}
	\end{split}
	\end{equation*} which is bounded by an absolute constant via the H\"{o}lder estimate \eqref{eq:flow} once we choose $\alpha \le e^{-cT}$ where $c$ is the constant from \eqref{eq:flow}. 
	
	Now we assume that $\omega(t,\cdot)$ is $C^{k,\alpha}$ in $\Phi(t,U_k)$ with some $k \ge 0$, where $U_k \supset \mathrm{cl}(O)$ and $d(U_k,A) > 0$. We first pick some open set $U_{k+1}$ which satisfies \begin{equation*}
	\begin{split}
	U_{k+1} \supset \mathrm{cl}(O), \qquad U_k \supset \mathrm{cl}(U_{k+1}).
	\end{split}
	\end{equation*} In particular, $U_{k+1}$ is separated away from $A$. For each $0\le t \le T$, take a smooth cutoff function $0 \le \chi_t \le 1$
	which equals 1 on $\Phi(t,U_{k+1})$ and vanishes outside of $\Phi(t,U_k)$. Then, for $x \in \Phi(t,U_{k+1})$, with the Biot-Savart kernel $K$, we write \begin{equation*}
	\begin{split}
	u(t,x) = (K*\omega_t)(x) = K*(\chi_t \omega_t )(x) + K*((1-\chi_t)\omega_t)(x).
	\end{split}
	\end{equation*} Regarding the first term, a classical singular integral estimate gives \begin{equation*}
	\begin{split}
	\V K* (\chi_t \omega_t ) \V_{C^{k+1,\alpha}(\mathbb{T}^2)} \le  C \V \nabla^k(\chi_t\omega_t)\V_{C^\alpha(\Phi(t,U_{k}))}.
	\end{split}
	\end{equation*} The second term is indeed $C^\infty$ in $\Phi(t,U_{k+1})$ simply because $K(\cdot)$ is $C^\infty$ away from the origin. Hence we deduce that $u(t,\cdot)$ is uniformly $C^{k+1,\alpha}$ in $\Phi(t,U_{k+1})$. At this point we may extend $u(t,\cdot)$ to be $C^{k+1,\alpha}$ on the entire domain $\mathbb{T}^2$ to obtain $\tilde{u}(t,\cdot)$. Then solving \begin{equation*}
	\begin{split}
	\frac{d}{dt} \tilde{\Phi}(t,x) = \tilde{u}(t,\tilde{\Phi}(t,x)), 
	\end{split}
	\end{equation*} gives that $\tilde{\Phi}(t,x)$ is a $C^{k+1,\alpha}$ flow, which coincides with $\Phi(t,x)$ whenever $x \in \Phi(t,U_{k+1})$ and $ 0 \le t \le T$. This can be done by obtaining an a priori estimate for $\V \tilde{\Phi}(t,\cdot)\V_{C^{k+1,\alpha}}$ and then argue along a (smooth) sequence of approximate solutions.  Note that \begin{equation*}
	\begin{split}
	|\nabla\tilde{\Phi}(t,x)| \ge \exp(-\int_0^t \V \nabla \tilde{u}(\tau,\cdot)\V_{L^\infty} d\tau )
	\end{split}
	\end{equation*} so $\tilde{\Phi}(t,\cdot)$ is invertible and the inverse function theorem gives that $\tilde{\Phi}_t^{-1}(\cdot)$ is also a $C^{k+1,\alpha}$ diffeomorphism of the domain. From \begin{equation*}
	\begin{split}
	\omega(t,z) = \omega_0(\Phi_t^{-1}(z)) = \omega_0(\tilde{\Phi}_t^{-1}(z)), \qquad z \in \Phi(t,U_{k+1}),
	\end{split}
	\end{equation*} differentiating both sides $k+1$ times, on the right hand side we obtain terms which contains up to the $k+1$th derivatives of $\omega_0$ (composed with $\Phi_t^{-1}$) multiplied with some factors of $\Phi_t^{-1}$ also up to the $k+1$th derivatives. Since each such factor is $C^\alpha$, we conclude that $\omega(t,\cdot)$ is $C^{k+1,\alpha}$. This finishes the proof. 
\end{proof}

\section{Open Problems}

In this section we discuss a few interesting open problems.

\subsection{Problem 1: Further Degeneration of Weak Solutions}

We have shown that there are Yudovich solutions of the $2D$ Euler equations on $\mathbb{T}^2$ which are initially in the class $H^1$ but which do not belong to the class $L^\infty( [0,\delta); H^1)$ for any $\delta>0$. 
One could ask whether even worse behavior is possible. In fact, the existing estimates do not rule out the existence of Yudovich solutions with $H^1$ data which leave $W^{1,p}$ for every $p>1$ in finite time. 

\begin{lemma}
	\label{LosingEstimate}
	Let $\omega_0 \in H^1(\mathbb{T}^2)\cap L^\infty(\mathbb{T}^2)$ be mean-zero. Then, the unique Yudovich solution satisfies the following estimate:
	$$|\omega(t)|_{W^{1,p}}\leq |\omega_0|_{H^1}$$
	for every $p\leq q(t),$ with $q(t)$ solving the ODE: $$\frac{d}{dt} q(t)= -Cq(t)^2|\omega_0|_{\infty},\quad q(0)=2$$ for some large universal constant $C$. 
\end{lemma}

\begin{proof}
	By the John-Nirenberg lemma, $$\int_{\mathbb{T}^2} e^{c|\nabla u|} dx \le C |\omega|_{\infty} \le C |\omega_0|_{\infty}$$ for some small universal constant $c$.   By a generalized Young's inequality, 
	\begin{equation*}
		\begin{split}
			\Vert \,|\nabla u| |\nabla \omega|^r \,\Vert_{1} \leq r\|e^{|\nabla u|}\|_{1} \cdot \| \,|\nabla \omega|^r \ln|\nabla \omega| \, \|_1.
		\end{split}
	\end{equation*} Now, by passing to the Lagrangian formulation (and suppressing the composition with the flow maps) we see:
	$$\partial_t \nabla\omega=-\nabla u\nabla\omega.$$
	Hence, \begin{equation*}
		\begin{split}
			\partial_t (|\nabla\omega|^{p(t)}) &= p'(t) |\nabla \omega|^{p(t)} \ln |\nabla \omega|+p(t)\partial_t \nabla\omega\cdot\nabla\omega  |\nabla\omega|^{p(t)-2} \\
			&=p'(t) |\nabla \omega|^{p(t)} \ln |\nabla \omega|-p(t) \nabla u \nabla\omega\cdot\nabla\omega |\nabla\omega|^{p-2}.
		\end{split}
	\end{equation*}
	Now upon integrating and using our inequality for $\big\||\nabla u| |\nabla\omega|^r\big\|_{1}$, we have:
	$$\frac{d}{dt} |\nabla\omega|_{p(t)}^{p(t)} \leq p'(t) \int |\nabla \omega|^{p(t)} \ln |\nabla \omega| dx + C  |\omega_0|_{\infty} p(t)^2\int |\nabla\omega|^{p(t)} \ln |\nabla\omega
	|dx.$$  Finally, choosing $$p'(t)= -C|\omega_0|_{L^\infty} p(t)^2,\qquad p(0) = 2,$$ we get
	$$\frac{d}{dt} \left(|\nabla\omega|_{p(t)}^{p(t)} \right)\leq 0.$$
\end{proof}

\begin{conj}
	The bound in Lemma \ref{LosingEstimate} is sharp, in the sense that there exist Yudovich solutions which continuously lose regularity. 
\end{conj}

It seems that proving the conjecture is true (for a short time) is much more difficult on $\mathbb{T}^2$ than on a domain with a boundary.

\subsection{Problem 2: Ill-posedness in $W^{2,1}?$}

Though there have recently been numerous results on ill-posedness for the Euler equations in critical spaces, it seems as though the case of $W^{2,1}$ vorticity (or $W^{3,1}$ velocity) is still open.

\subsection{Problem 3: Vanishing Viscosity}

Consider the $2D$ Euler equations with partial viscosity on $\mathbb{T}\times [0,1]$:
$$\partial_t \omega + u \cdot \nabla\omega = \nu \partial_{x_1x_1}\omega.$$
Notice that we have put viscosity only in the horizontal variable and, hence, we only need the no-slip boundary condition: $u^2=0$ on $x_2=0$ and $x_1=1$.

When $\nu=0$, we see that $H^1$ data can leave $H^1$ initially using a modification of the proof of Theorem \ref{thm:pat}. In fact, it can be shown that $|\partial_{x_1}\omega|_{2}$ becomes infinite. When $\nu>0$ this is no longer possible due to the energy equality:
$$\frac{d}{dt} |\omega|_{2}^2 = -2\nu \int |\partial_{x_1}\omega|^2 dx.$$

There a few regimes where one could study the behavior of the solutions of the partially viscous problem as $\nu\rightarrow 0$. The first regime is when $\nu,t\rightarrow 0$. Depending upon the relative sizes of $\nu$ and $t$, different behaviours can be observed. In particular, one would expect that if $\nu\ll t$ that we could see $H^1$ growth immediately. On the other hand, if $t\ll\nu$ we shouldn't see any growth. Determining the exact dynamics in this regime seems interesting. By the same token, one could consider the inhomogenous problem and study the limit $t\rightarrow \infty$ and $\nu\rightarrow 0$.

\section*{Acknowledgments}

We thank the anonymous referee for numerous suggestions and comments, which have significantly improved the quality of this article.

\end{document}